\documentclass[a4paper,11pt, reqno]{amsart} 
\usepackage{mathptmx}
\usepackage{tabularx}
\usepackage{booktabs}
\usepackage{multirow}
\usepackage{amsmath}
\usepackage[T1]{fontenc}
\usepackage{lmodern}
\usepackage[utf8]{inputenc}
\usepackage{mathtools} 

\usepackage{amssymb}
\usepackage{amsfonts}
\usepackage{fancyhdr}
\usepackage{graphicx}
\usepackage{hyperref}
\usepackage{enumerate}
\usepackage{amsthm}
\usepackage{tikz}
\usetikzlibrary{patterns}

\usepackage[margin=1in]{geometry}

\newtheorem{theorem}{Theorem}[section]

\newtheorem{proposition}[theorem]{Proposition}

\newtheorem{definition}[theorem]{Definition}

\newcommand{\R}{\mathbb R}
\newcommand{\Z}{\mathbb Z}

\newcommand{\si}{\sigma}

\newcommand{\la}{\lambda}
\newcommand{\C}{\mathbb C }

\newcommand{\pa}{\partial }

\newcommand{\Hc}{\mathcal H}

\newcommand{\om}{ \omega}

\newcommand{\D}{\Delta}
\newcommand{\ap}{\alpha}
\newcommand{\bt}{\beta}

\newcommand{\Gm}{\Gamma}
\newcommand{\gm}{\gamma}

\newcommand{\re}{\operatorname{Re}}
\newcommand{\im}{\operatorname{Im}}

\newcommand{\ds}{\displaystyle}
\newcommand{\sar}{S^a_{rad}}
\newcommand{\sa}{S^a}
\usepackage{tikz}

\title[Characterization of eigenfunctions of the Laplacian having exponential growth]
{Characterization of eigenfunctions of the Laplacian having exponential growth}

\author[Basil Paul, Pradeep B]{Basil Paul, Pradeep Boggarapu}
\address[Basil Paul]{Department of Mathematics\\
		BITS Pilani K K Birla Goa Campus\\
		Zuarinagar, South Goa\\
		403 726, Goa, India}
\email{basilpaul9192@gmail.com}

\address[Pradeep B.]{Department of Mathematics\\
		BITS Pilani K K Birla Goa Campus\\
		Zuarinagar, South Goa\\
		403 726, Goa, India}
\email{pradeepb@goa.bits-pilani.ac.in}

\keywords{Laplacian, Spherical Fourier transform, Eigenfunction, Euclidean spherical functions}
\subjclass[2020]{Primary: 42B10; Secondary: 46E10.}

\begin{document}

	\begin{abstract}
		In 1993, Robert Strichartz proved a characterization for the bounded eigenfunctions of  Laplacian $\D=-\sum_{j=1}^d \frac{\pa^2}{\pa x_j^2} $ on $\R^d$: If $\left\{f_k \right\}_{k\in \Z}$ be a doubly infinite sequence of functions on $\R^d$ such that $\Delta f_k=f_{k+1}$  and $ \|f_k\|_{L^{\infty}(\R^d)} \leq C$ for all $ k \in \Z$, for some $C>0$,  then $f_0$ is an eigenfunction of $\Delta$. Observing the existence of unbounded eigenfunctions of Laplacian, Howard, and Reese generalized Strichartz's theorem to characterize eigenfunctions of Laplacian having at most polynomial growth. In this article, we shall prove an extended version of Strichartz's theorem to characterize eigenfunctions of the Laplacian having exponential growth.
	\end{abstract}
	\maketitle
	\section{Introduction}\label{section1}

	It is worthwhile to observe that a function on the real line with the property that all its derivatives and anti-derivatives are uniformly bounded must be a linear combination of $\sin x$ and $\cos x$. Taking this fact into account, John Roe \cite{Roe} proved the following characterization for sine functions $a sin(x + q)$ in terms of the size of their derivatives and anti-derivatives.
	
	\begin{theorem}[Roe] \label{th:roe}
		Let $\left\{f_{n}\right\}_{n=-\infty}^{\infty}$ be a doubly infinite sequence of real-valued functions of a real variable with
		\[f_{n+1}(x)=\frac{d}{d x} f_{n}(x).\]
		Suppose further that there exists a real $M$ such that $|f_{n}(x)| \leq M$  for all $ n \in \Z$ and $x \in \R.$ Then, $f_{0}(x)=a \sin (x+\phi)$ for some real constants $a$ and $\phi$. 
	\end{theorem}
	
	In the article \cite{howard}, Ralph Howard proved a general version of Roe's theorem where, it was  showed that the bounds $\left|f_{n}(x)\right| \leq M$ can be relaxed to $\left|f_{n}(x)\right| \leq M_n(1+|x|)^m$ with $m\geq 1$ and where the constants only need to have subexponential growth.
	\begin{theorem}[Howard]\label{th:howard1} 
		Let $\left\{f_{n}\right\}_{n=-\infty}^{\infty}$ be a doubly infinite sequence of complex valued functions defined on the real numbers with
		$$ f_{n+1}(x)=\frac{d}{d x} f_{n}(x) $$
		and so that there are constants $M_n \geq 0, \alpha \in[0,1)$, and a nonnegative integer $k$ satisfying
		$$ \left|f_{n}(x)\right| \leq M_n(1+|x|)^{k+\alpha}.	$$
		
		If	$$	\varliminf_{n \rightarrow \infty} \frac{M_n}{(1+\varepsilon)^n}=0 \quad \text {for all } \varepsilon>0	$$
		and
		$$ \varliminf_{n \rightarrow \infty} \frac{M_{-n}}{(1+\varepsilon)^n}=0 \quad \text {for all } \varepsilon>0 , $$	then
		$$	f_{0}(x)=p(x) e^{i x}+q(x) e^{-i x}	$$
		where $p(x)$ and $q(x)$ are polynomials of degree at most $k$.
	\end{theorem}
	
	Let $\D = - \sum_{j=1}^{d} \frac{\partial^2}{\partial x^2_j}$ be the Laplacian operator on $\R^d$. In 1993, Robert Strichartz \cite{Str} studied a $d$-dimensional generalization of Roe's theorem \cite{Roe}, where $\frac{d}{dx}$ is replaced by the Laplacian on $\mathbb{R}^d$ and a characterization of bounded eigenfunctions of the Laplacian with eigenvalue $\ap>0$ was obtained.
	
	\begin{theorem}[Strichartz]\label{th:str}
		Let $\left\{f_k \right\}_{k\in \Z}$ be a doubly infinite sequence of functions on $\R^d$ satisfying
		$\Delta f_k=\alpha f_{k+1}$ for some $\alpha$>0, for all $k \in \Z$. If $ \|f_k\|_{L^{\infty}(\R^d)} \leq C$ for all $ k \in \Z$, for some $C>0$ ,  then $\Delta f_0=\alpha f_0$.
	\end{theorem}
	
	Strichartz also remarked the above result can be extended for $L^p$ norms, where $p> \frac{2d}{d-1}$ instead of the $L^\infty$ norm. Meanwhile, observing the existence of eigenfunctions of Laplacian which are unbounded, Howard and Reese \cite{Howard} extended the above theorem to characterize eigenfunctions of Laplacian having at most polynomial growth.

	\begin{theorem}[Howard and Reese]\label{th:howard2}
		Let $a \geq 0$ and let $\{ f_k\}_{k=-\infty}^{\infty}$ be a doubly infinite sequence of complex-valued functions on $\mathbb{R}^d$ that satisfy
		$$ \Delta f_k=f_{k+1} $$
		and
		\begin{equation}\label{eq:pg} \left|f_k(x)\right| \leq M_k(1+|x|)^a \end{equation}
		where the constants $M_k$ have sublinear growth:
		$$ \lim _{k \rightarrow \infty} \frac{M_k}{k}=\lim _{k \rightarrow \infty} \frac{M_{-k}}{k}=0
		$$
		Then $\Delta f_0=f_0$.
	\end{theorem}
	
	In the same article \cite{Howard}, the authors mentioned how unlikely it is to characterize eigenfunctions of Laplacian having exponential growth, which is the primary motivation behind this work. As an example of such eigenfunctions, consider two important class of functions - first one being known as the Euclidean spherical functions denoted by $\phi_{\la}$ and defined by 
	
	$$	\phi_\la(x) = \int_{S^{d-1}} e^{i\la x.\om}d\si(\om)  $$

	where $d\si$ is the normalized surface measure on the unit sphere $S^{d-1}$ and $\la \in \C$ and the second one being defined as  $ \ds E_\zeta (x)= e^{i\zeta \cdot x}$ for $\zeta=(\zeta_1, \ldots, \zeta_d) \in \C^d$. One can verify that for $\la \in \C$ and $\zeta \in \C^d$  both these family of functions are eigenfunctions of the Laplacian having exponential growth i.e., they satisfy,
$$|\phi_\la(x)|\leq e^{|\im(\la)||x|}~~\text{and}~~|E_\zeta(x)|\leq e^{|\im(\zeta)||x|}.$$
But they don't satisfy the condition \eqref{eq:pg} as mentioned in Theorem \ref{th:howard2} if $\im(\la)\ne 0$ and $\im(\zeta)=(\im(\zeta_1), \ldots, \im(\zeta_d))\ne 0$.
	
	Precisely, our aim is to ascertain if a doubly infinite sequence of complex-valued functions $\left\{f_k \right\}_{k\in \Z}$ on $\R^d$ satisfies the  conditions $ \Delta f_k= \ap f_{k+1} $ and 
	\begin{equation}\label{eq:eg} | f_k(x) | \leq M e^{a|x|}, \end{equation}for all $k \in \Z$ and $x\in \R^d$, then $f_0$ is an eigenfunction of $\D$. 
	
	As a prelude before verifying the above objective, we must necessarily have an understanding regarding the point spectrum of the Laplacian in the function space $X_a$, where 
$$X_a = \{f : \R^d \to \C |~ |f(x)| \leq Me^{a|x|}\}.$$ for a fixed $a>0$.
It is into this context we use a crucial observation that $\phi_\la \in X_a $ for $|\im(\la)|\leq a$ and for any function $f \in X_a$ such that $ \D f= \la^2 f$, then $|\im(\la)| \leq a$ (See Proposition \ref{proposition 2.1}). Hence we shall see that the point spectrum of Laplacian on $X_a$ is precisely $\Lambda(\Omega_{a})$, where $\Lambda$ denotes the map $\la \mapsto \la^2$ and $\Omega_{a}$ denotes the complex strip $$\{\la \in \C |   |\im(\la)| \leq a\}.$$
 Further, routine calculations shows that the boundary $\partial \Lambda(\Omega_{a})$ of $\Lambda(\Omega_{a})$ is a parabola given by $\{\la=s+it|t^2 = 4a^2(s+a^2)\}$, see Figure 1. Keeping in mind the geometry of the point spectrum of Laplacian, we shall aim to characterize eigenfunctions of Laplacian having exponential growth corresponding to any arbitrary complex eigenvalue.
	
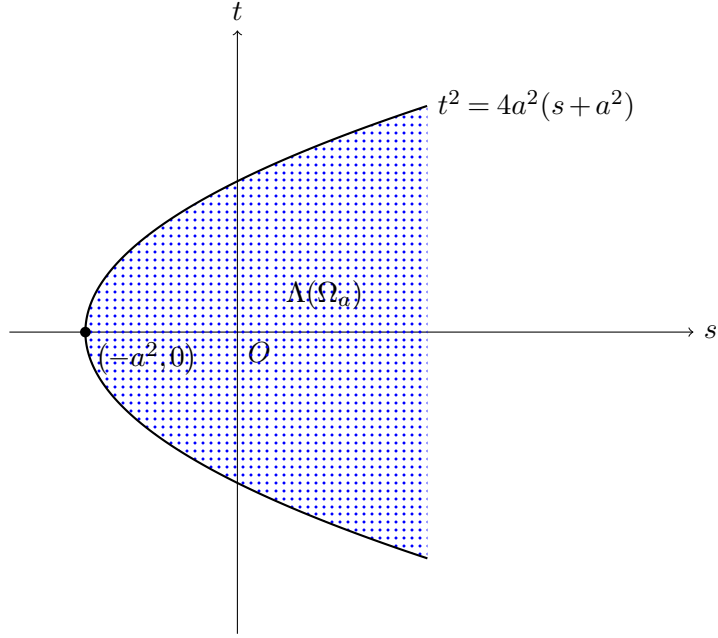
\begin{figure}[h!]
    \centering
	\begin{tikzpicture} \label{Diagram 1}
				\draw[->] (-3,0) -- (6,0) node[right] {$s$};
				\draw[->] (0,-4) -- (0,4) node[above] {$t$};
				
				\node[below right] at (0,0) {$O$};
                                          \fill[pattern=dots, pattern color=blue]
    plot[domain=-3:3] ({\x*\x/2-2}, \x) -- 
    plot[domain=-3:3] (2.5, \x);
\node at (0.5, 0.5) [right]{$\Lambda(\Omega_a)$};
				
				\draw[thick,domain=-3:3,smooth,samples=100]
				plot ({\x*\x/2-2}, \x)
				node[right] {$t^2=4a^2(s+a^2)$};
				\fill[black] (-2,0) circle(2pt);
				\node at (-2, 0) [below right] {$\left(-a^2, 0 \right)$};				
	\end{tikzpicture}
    \caption{Point spectrum of $\D$ on $X_a$.}
\end{figure}
 For any $\lambda_0\in \C$ with $\im (\la_0)\ne 0$ and $a_{\la_0}=|\im(\la_0)|$, it is easy to see that $\la_0^2 \in \partial \Lambda(\Omega_{a_{\la_{0}}})$. 
We formulated a characterization for the eigenfunctions of Laplacian having exponential growth corresponding to the complex eigenvalue $\la_{0}^2$ as follows.  
\begin{theorem}\label{Theorem1.5}
	For any $\la_0 \in \C$ with $\im(\la_0) \neq 0$,  let $N(\la_0)$ denotes the outward normal drawn to the parabola $\partial \Lambda(\Omega_{|\im(\la_0)|})$ at the point $\la_{0}^2$.  Let $z_0 \in N(\la_0)$ and $\{f_k\}_{k \in \mathbb{Z}}$ be a doubly infinite sequence of functions on $\R^d$ satisfying 
	\begin{enumerate}
		\item $(\D- z_0 I)f_k = Af_{k+1}$ for some non-zero $A \in \C $ and
		\item  $ |f_k(x)| \leq M e^{|\im(\la_0)| |x|} $ for a constant $M>0$ and for every $x \in \R^d$. 
	\end{enumerate} Then the following assertions hold.
	\begin{enumerate}[(a)]
		\item If $|A| = |{\la_0}^2 - z_0|$, then $\D f_0 = {\la_0}^2 f_0$.
		\item If $|A| < |{\la_0}^2 - z_0|$, then $f_k = 0,$ for all $k \in \Z$.
		\item There are solutions satisfying conditions (1) and (2) which are not eigenfunctions of $\D$ when $|A| >  |{\la_0}^2 - z_0|$.
	\end{enumerate}
\end{theorem}

%
	
	Further, as an extension of the Roe's theorem \ref{th:roe} and Howard's theorem \ref{th:howard1}, we shall prove the following result, which is a Roe's type characterization of exponential function. 
	\begin{theorem}\label{Theorem 1.2}
		\par 	For any $\la_0 \in \C$ with $\im(\la_0) \neq 0$, let $\{f_k\}_{k \in \mathbb{Z}}$ be a doubly infinite sequence of functions on $\R$ satisfying 
		\begin{enumerate}
			\item For some non-zero $A \in \C $, we have $\begin{cases}
				(\frac{d}{dx} - i\la_0 + \ap)f_k = Af_{k+1}, &\mbox {if}~ \im(\la_0)>0\\ (\frac{d}{dx} - i\la_0 - \ap)f_k = Af_{k+1}, &\mbox {if}~ \im(\la_0)<0,
			\end{cases}$
			\item  $ |f_k(x)| \leq M e^{|\im(\la_0)| |x|} $ for a constant $M>0$ and for every $x \in \R$. 
		\end{enumerate} Then the following assertions hold.
		\begin{enumerate}[(a)]
			\item If $|A| = |\ap|$, then $\frac{d}{dx} f_0 = i{\la_0} f_0$. i.e. $f_0(x)= e^{i\la_0 x}$.
			\item If $|A| < |\ap|$, then $f_k = 0,$ for all $k \in \Z$.
			\item There are solutions satisfying conditions (1) and (2) which are not eigenfunctions of $\frac{d}{dx}$ when $|A| >  |\ap|$.
		\end{enumerate}
	\end{theorem}
	
	Having stated our aim, our strategy is to formulate a general version of the above results, so that these results happens to be  consequential. The organisation of this paper is as follows. We shall devote the entire Section \ref{section2} for setting up the necessary framework and establishing our main results. In particular, we shall define a  \textit{Schwartz type space}  $S^a(\R^d)$ and its corresponding dual space named as the space of all \textit{exponential tempered distributions of type $a$}, adapting the definition from the classical case. It is already a known fact that the spherical Fourier transform as defined in Subsection \ref{subsection2.2} is a topological isomorphism from the space of all radial functions in $S^a(\R^d)$ to the space of all even holomorphic functions which are rapidly decaying at infinity on the strip $\Omega_{a}$. Further, this isomorphism can be extended to the corresponding dual spaces as well and using this, we shall prove a version of the Strichartz's theorem for characterizing the exponential type tempered eigendistributions of the Laplacian. Then the proof of  Theorem \ref{Theorem1.5} follows by establishing the fact that functions having exponential growth are indeed exponential tempered distributions of type $a$. Our proof for the main theorem closely follows the methodologies adopted by the authors M. Naik and R. Sarkar \cite{muna} for characterizing eigenfunctions of the Laplace-Beltrami operator on Riemannian symmetric spaces of non-compact type with real rank one. Finally, in Section \ref{section3} we shall conclude this article by establishing a Roe's type characterization for the exponential type tempered eigendistributions of the ordinary derivative $\frac{d}{dx}$ as in Theorem \ref{Theorem 4.1}, from which Theorem \ref{Theorem 1.2} follows. 
	
	\section{Characterization of the exponential type tempered eigendistributions of the Laplacian}\label{section2}
	\subsection{Setting the framework}\label{subsection2.1}
	\subsubsection{Schwartz type space $S^a(\R^d)$ and $H_e(\Omega_{a})$}:

	We shall introduce a \textit{Schwartz type space} as follows. For a fixed real number $a \geq 0$, let $S^a(\R^d)$ be the space of all $C^{\infty}$ functions on $\R^d$ such that 
	
	$$ \gm _{m,\ap}(f) = \sup_{x \in \R^d} e^{a|x|} (1+|x|)^m |D^\ap f(x)| < \infty$$
	for every non-negative integer $m$ and for all multi-index $\ap=(\ap_1, \ldots, \ap_d)$, where $$\displaystyle D^\ap f = \frac{\partial^{|\ap|}f}{{\partial x_1}^{\ap_1}\cdots \partial {x_d}^{\ap_d}}.$$ We note that for $ a=0$, above space coincides with the standard Schwartz space $\mathcal{S}(\R^d)$. Let $\sar(\R^d)$ denote the space of all radial functions in $S^a(\R^d)$. We now define the strip $\Omega_{a} = \{\la \in \C |  |\im(\la)| < a\}$. Let $H_e(\Omega_{a})$ be the set of all even holomorphic functions on $\Omega_{a}$ which are continuous on the closure of $\Omega_{a}$ and satisfying for all non-negative integers $m$ and for all multi-indices $\ap$ 
	$$\mu_{m,\ap}(\phi) = \sup_{\la \in \Omega_{a}} (1+|\la|)^m \left|\left(\frac{d}{d\la}\right)^\ap \phi (\la)\right|< \infty. $$
	
	It can be seen that both spaces $S^a(\R^d)$ and $H_e(\Omega_{a})$ are Fr\'echet spaces with respect to the topology induced by the family of seminorms $\{\gm _{m,\ap}\}$ and $\{\mu_{m,\ap}\}$ respectively. Now for any function $f\in \sar(\R^d)$ we define the spherical Fourier transform of $f$ as
	$$ \mathcal{H}f(\la)= \int_{\R^d} f(x) \phi_{\la}(x)dx. $$
It is known that, $\mathcal{H}: \sar(\R^d)\to H_e(\Omega_{a})$ is a topological isomorphism, see \cite[Theorem 2.1]{natan}. Using classical Fourier inversion formula, we can show that 
$$\mathcal{H}^{-1}\phi(x)=\int_{0}^\infty \phi(\la)\phi_\la(x) \la^{d-1}d\la.$$

	\subsubsection{Extension of spherical Fourier transform to corresponding dual spaces $S^a (\R^d)'$ and $H_e(\Omega_{a})'$}\label{subsection2.1.2} 

	In what follows $H_e(\Omega_{a})'$ denotes the dual of $H_e(\Omega_{a})$ and for  any $S \in H_e(\Omega_{a})'$, for any suitable even holomorphic function $\psi$ on $\Omega_a$,  we define $ \psi S$ as a dual element by the equation $$\langle \psi S, \phi \rangle = \langle S,  \psi \phi \rangle, $$ for all $\phi \in H_e(\Omega_{a})$,  where $\langle . , . \rangle$  denotes the dual bracket.

	Now we name $S^a(\R^d)'$, the dual space of $S^a(\R^d)$ as the space of  all exponential tempered distributions of type $a$ on $\R^d$ and for any $\psi \in S^a(\R^d)$, we define the action of the operator $\D$ on $T \in  S^a(\R^d)'$ as follows:  $\langle \D T, \psi \rangle = \langle T,  \D \psi \rangle $.
	Let $T$ be an exponential tempered distribution of type $a$. The spherical Fourier transform of $T$, denoted by $\mathcal{H}T$ is defined as a linear functional on $H_{e}(\Omega_{a})$ by the following rule:
	$$ \langle \mathcal{H}T, \phi \rangle = \langle T,\mathcal{H}^{-1}\phi \rangle $$
	where $\phi \in H_{e}(\Omega_{a}), \mathcal{H}^{-1}\phi \in S^{a}(\R^{d})$ and $\mathcal{H}(\mathcal{H}^{-1}\phi)= \phi$.
	
	\subsubsection{Radialization operator}
	For a suitable function $f$ on $\R^d$, its radialization $Rf$ is defined as $$ Rf(x) = \int_{S^{d-1}} f(|x|\omega) d\sigma(\omega)$$
	 where $d\sigma$ denote the normalized surface measure on $S^{d-1}$. Let us have a look at some of the important properties of the radialization operator:
	 \begin{itemize}
	 	\item $$\int Rf(x)g(x) dx  = \int f(x) Rg(x) dx , $$where $f, g \in S^a(\R^d)$;
	 	\item $ R(\D f)= \D (Rf)$. 
	 \end{itemize}
	 The above definition of radialization can be extended to $T \in \sa(\R^d)' $ as follows.
	 \begin{definition}
	 	For any $T \in \sa(R^d)'$, the radialization of $T$ is defined by the rule $\langle RT, f \rangle = \langle T, Rf \rangle$ for all $f \in \sa(\R^d)$. Thus $T \in \sa(\R^d)'$ is said to be radial if it satisfies $RT = T$.
	 	
	 \end{definition}

	 \subsubsection{Translation operator}
	 Given any $y \in \R^d$ and a function $f: \R^d \to \C,$ we define the translation of $f$ as $(\ell_yf)(x)= f(x-y)$.  For $T \in S^a(\R^d)'$ and $y \in \R^d$ , we define the translation $\ell_yT$ via $$\langle \ell_yT, f \rangle= \langle T, \ell_{-y} f \rangle, $$ for every $f \in S^a(\R^d)$.
	 
	 \subsubsection{Point spectrum of Laplacian in $X_a$}
	 Let us have a look at the following proposition which is a crucial ingredient in the formulation of our main theorem \ref{Theorem 2.1}, as it provides an insight into the point spectrum of the Laplacian in the space $X_a$, where 
	 $$X_a = \{f : \R^d \to \C |~ |f(x)| \leq Me^{a|x|}\}.$$ 
	 for a fixed $a>0$.
	 \begin{proposition}\label{proposition 2.1}
	 	Let $f$ be a non-zero function in $X_a$ such that $\D f = \la^2 f$, then $ |\im(\la)| \leq a$.
	 \end{proposition}
	 
	 \begin{proof} Note that $f$ is an eigenfunction of $\D$ implies that its radialization $Rf$ is also an eigenfunction of $\D$, as $\D$ commutes with radialization $R$.  Let $Rf(x)=f_0(|x|)$ for some function $f_0$ on $[0, \infty)$, then $f_0$ satisfies 
\begin{equation}\label{e1}
	 		\frac{d^2f_0}{dr^2} + \frac{d-1}{r} \frac{df_0}{dr} + \la^2 f_0= 0.
	 	\end{equation}   On applying a suitable change of variable of the form $f_0(r)= (i\la r)^{-\frac{d}{2}+1} u(i\la r)$, \eqref{e1} can be converted to a Modified Bessel's equation of the form $$ r^2 u'' + r u'- \left[r^2 + \Big(\frac{d}{2}-1\Big)^2\right]u=0 $$ whose general solution is obtained as
	 	\begin{equation}\label{e4}
	 		u(r) =  A I_\nu( r) + B K_\nu( r)
	 	\end{equation}  where $ I_\nu$ and $ K_\nu$ are the modified Bessel functions and $\nu = \frac{d}{2}-1$ [Refer pages 223-226 of \cite{niki} for more details].  Now $f_0$ being bounded near $r=0$ implies the boundedness of $u$ at $r=0$ and thus $B=0$ in \eqref{e4}. Hence $$f_0(r) = A(i\la r)^{-\nu}  I_{\nu}(i\la r).$$Using the asymptotic behaviour of  $I_\nu(r)$ as $r\to \infty$, we obtain that $$ f_0(r)= A (i\la r)^{-\nu} \frac{e^{(i\la r)}}{\sqrt{2 \pi i \la r }} [1+ O(1/r)].$$ Since $|f_0(r)| \leq M e^{ar}$ for $r>0$, we must have that $|\im (\la)|\leq a$.

%
%
	 	 \end{proof}
In view of the above proposition, it can be concluded that the point spectrum of Laplacian on $X_a$ is precisely $\Lambda(\Omega_{a})$, where $\Lambda$ denotes the map $\la \mapsto \la^2$ and $\Omega_{a}$ denotes the complex strip $$\{\la \in \C |   |\im(\la)| \leq a\}.$$

	\subsection{Characterization of the exponential type tempered eigendistributions of the Laplacian}\label{subsection2.2}
	
	Under the framework defined in Subsection \ref{subsection2.1}, we shall now formulate our major result, which is a characterization of the exponential-type tempered eigendistributions of the Laplacian corresponding to any arbitrary complex eigenvalue.
	
Recall that for any $\la \in \C$, $N(\la)$ is the set of all points on the outward normal drawn to parabola $\partial \Lambda(\Omega_{a_{\la}})$ at the point $\la^2$ with $a_\la=|\im(\la)|$. With these notations, we shall state the theorem as follows.
	\begin{theorem}\label{Theorem 2.1}
		For any $\la_0 \in \C$ with $a=|\im(\la_0)| \neq 0$ and $z_0 \in N(\la_0),$ let $\{T_k\}_{k \in \mathbb{Z}}$ be a doubly infinite sequence of exponential tempered distributions of type $a$ satisfying 
		\begin{enumerate}
			\item $(\D- z_0 I)T_k = AT_{k+1}$ for some non-zero $A \in \C $ and
			\item  There exist a seminorm $\gm$ on    $S^a(\R^d)$ and a constant $M>0$ such that $|\langle T_k, f \rangle | \leq M \gm(f)$ for all $f \in S^a(\R^d)$ and for all $k \in \Z$. 
		\end{enumerate} Then the following assertions hold.
		\begin{enumerate}[(a)]
			\item If $|A| = |\la_0 ^2 - z_0|$, then $\D T_0 = \la_0 ^2 T_0$.
			\item If $|A| < |\la_0 ^2 - z_0|$, then $T_k = 0,$ for all $k \in \Z$.
			\item There are solutions satisfying conditions (1) and (2) which are not eigendistributions of $\D$ when $|A| >  |\la_0 ^2 - z_0|$.
		\end{enumerate}
	\end{theorem}
	
	\begin{proof}
		Proof of (a).	First, we assume that $T_k$'s are radial. Let $|A|=  |\la_0^2 - z_0|$ and $\la_0 ^2$ be such that $a:=\im(\la_0) > 0$. Observe that $|\Lambda(\la)- z_0| \geq | \Lambda(\la_0) - z_0|$, for $\la \in \Omega_a$. We note that $\mathcal{H}(\D T)=\la^2 \mathcal{H}T$.  From the hypothesis, we obtain $(\D - z_0I)^k T_0 = A^k T_k$ and hence $(\la^2 - z_0)^{k} \Hc{T_0} = A^k \Hc{T_k}$. \\The proof shall comprise two steps. In step (1), we shall show that \begin{equation}\label{eq03}
			(\la^2 - \la_0^2)^{N+1} \Hc{T_0} = 0 	\end{equation}  for some $N \in \Z ^+$  or equivalently $ \langle (\la^2 - \la_0^2)^{N+1} \Hc{T_0}, \phi \rangle = 0 $ for every $\phi \in H_e(\Omega_a)$. Therefore consider for any $\phi \in H_e(\Omega_a)$,
		\begin{align*}
			|\langle(\la^2 - \la_0^2)^{N+1}\Hc{T_0}, \phi \rangle|  &= |\langle \Hc{T_k}, \left(\frac{A}{\la^2 -z_0} \right)^k(\la^2 - \la_0^2)^{N+1} \phi \rangle|\\ &= \left | \bigg \langle T_k, \Hc^{-1}\bigg({\left(\frac{A}{\la^2 -z_0}\right)^k (\la^2- \la_0^2)^{N+1}\phi \bigg)} \bigg \rangle \right|  \\
			&\leq M \gm \bigg[\Hc^{-1}\bigg({\left(\frac{A}{\la^2 -z_0}\right)^k (\la^2- \la_0^2)^{N+1}\phi \bigg)}\bigg] \\
			&\leq   M' \mu\bigg[{\left(\frac{A}{\la^2 -z_0}\right)^k (\la^2- \la_0^2)^{N+1}\phi}\bigg] 
		\end{align*} where the seminorm $\mu$ is given by $$\mu(\phi)= \sup_{\la \in \Omega_{a}^+} \left| \frac{d^\tau}{d\la^\tau} P(\la)\phi(\la)\right|$$ for some even polynomial $P(\la)$ and derivative of even order $\tau$, where $\Omega_a^+=\{\la\in \Omega_a|\im(\la)\geq 0\}.$ Let $N = 6\tau + 1$ be fixed. Henceforth, we shall use the following notations. Let $F^k$ denotes the term:  $$\left| \frac{d^\tau}{d\la^\tau} P(\la)\bigg(\frac{A}{\la^2-z_0}\bigg)^k (\la^2-\la_0^2)^{N+1}\phi(\la)\right|$$
		We intend to show that $\sup_{\la\in\Omega_{a}^+} F^k \to 0,$ as $ k \to \infty$. 
		Note that 
\begin{multline*} \left| \frac{d^\tau}{d\la^\tau} P(\la)\bigg(\frac{A}{\la^2-z_0}\bigg)^k (\la^2-\la_0^2)^{N+1}\phi(\la)\right| = \sum_{\substack{l+m+n= \tau \\ l,m,n \in \Z^+}} C_{lmn} \frac{d^l}{d\la^l} \bigg(\frac{A}{\la^2-z_0}\bigg)^k \\ \times \frac{d^m}{d\la^m} (\la^2-\la_0^2)^{N+1} \frac{d^n}{d\la^n} (P(\la)\phi).\end{multline*}

		Using the facts that $ \phi \in H_e(\Omega_{a})$ and $ \bigg|\frac{A}{\la^2 - z_0} \bigg| = \bigg | \frac{\la_0^2- z_0}{\la^2-z_0}\bigg| \leq 1 $, for $ \la \in \Omega_{a}^+$, we can have the following two inequalities.
		\begin{equation}\label{eq04}
			F^k(\la)  \leq C_1 k^\tau \bigg | \frac{\la_0^2- z_0}{\la^2-z_0}\bigg|^k |\la^2- \la_0^2|^{N+1-\tau}	\end{equation}  and 	\begin{equation}\label{eq05}
			  F^k(\la)  \leq C_2 k^\tau \bigg | \frac{\la_0^2- z_0}{\la^2-z_0}\bigg|^k
		\end{equation}
	 where  $ C_1$ and $C_2$ being constants.
		Now we shall choose a compact connected neighborhood $\mathcal{U}$ of $\la_0$ in $\Omega_{a}^+$ in such a way that $\bigg |\frac{\la_0^2 - z_0}{\la^2-z_0}\bigg| < \frac{1}{2}$, for $\la \not\in  \mathcal{U}$. It follows from \eqref{eq05} that $F^k \to 0$ uniformly, as $ k \to \infty$ on $\Omega_{a}^+ \setminus \mathcal{U}$.
		It remains to show that
		 $$\sup_{\la \in \mathcal{U}} k^\tau \bigg | \frac{\la_0^2- z_0}{\la^2-z_0}\bigg|^k |\la^2- \la_0^2|^{5\tau+2} \to 0,$$ as $ k\to \infty$. On elucidating further, we shall see that \begin{equation}\label{eq06}
		 	 \sup_{\la \in \mathcal{U}} k^\tau \bigg | \frac{\la_0^2- z_0}{\la^2-z_0}\bigg|^k |\la^2- \la_0^2|^{5\tau+2} = \sup_{\la \in \mathcal{U}} k^\tau \bigg | \frac{\Lambda(\la_0)- z_0}{\Lambda(\la)-z_0}\bigg|^k |\Lambda(\la)-z_0 - (\Lambda (\la_0)-z_0)|^{5\tau+2}\end{equation}
		 	 From \eqref{eq06} and a detailed analysis of the Figure 2 implies that it is enough for us to prove that $$\sup_{z \in \Gamma } k^\tau \left|\frac{\beta}{z}\right|^k |\beta-z|^{5\tau+2} \to 0$$ where $\Gm$ is a compact region containing $\beta$, bounded by the parabolic arc and vertical line, lying on one side of tangent drawn at $\beta$ opposite to the origin. Now we may assume that $\beta$ lies on positive imaginary axis by applying a suitable rotation. We shall now show that $\ds\sup_{z \in H} k^\tau \left|\frac{\beta}{z}\right|^k |\beta-z|^{5\tau+2} \to 0$, as $k \to \infty$ where $H = \{z \in \C| -\eta \leq \re z \leq \eta, |\beta| \leq \im z\leq \delta\}$ for any $\eta > 0$ and $ \delta > |\beta|$. Let $V_k = \{z \in H | ~|\re(z-\bt)|< k^{\frac{-1}{4}}, | \im(z-\bt)| < k^{\frac{-1}{4}}\}$ and $ V_k^c = H \setminus V_k.$ Thus it can be verified that if $z \in V_k^c,$ then $ |z| \geq (|\bt|^2 + k^\frac{-1}{2})^{\frac{1}{2}}$. Hence for $ z \in V_k^c$, we have
		$$ \left|\frac{\bt}{z}\right| \leq \frac{|\bt|}{(|\bt|^2 + k^{\frac{-1}{2}})^{\frac{1}{2}}} = \left(1+ \frac{c_1}{\sqrt{k}}\right)^{\frac{-1}{2}}$$
		where $c_1 = |\beta|^{-2}$. Now using the compactness of $H$, we can see that for some constant $c_3,$
		\begin{equation}\label{eq07}
			\sup_{z \in V_k^c} k^\tau \left|\frac{\beta}{z}\right|^k |\beta-z|^{5\tau+2} \leq c_3 \left(1+ \frac{c_1}{\sqrt{k}}\right)^{\frac{-k}{2}} 	
			\end{equation} 
		
		If $z \in V_k$, then $|\re(z-\bt)| < k^{\frac{-1}{4}}$ and $|\im(z-\bt)|<k^{\frac{-1}{4}}$. Therefore, we obtain a constant $c_4$ such that
		\begin{equation}\label{eq08}
		 \sup_{z \in V_k } k^\tau \left|\frac{\beta}{z}\right|^k |\beta-z|^{5\tau+2} \leq c_4 k^\tau k^{\frac{-(5\tau + 2)}{4}}= c_4 k^{\frac{-(\tau+2)}{4}} 	\end{equation}
		Above inequalities \eqref{eq07} and \eqref{eq08} implies that $\sup_{z \in H} k^\tau |\frac{\beta}{z}|^k |\beta-z|^{5\tau+2} \to 0$ as $k \to \infty$. Hence we established that $F_k(\la) \to 0 $ uniformly, as $k \to \infty$.\\

\begin{figure}[h!] \label{Diagram 2}
\begin{center}
	\begin{tikzpicture}
				\draw[->] (-4,0) -- (5,0) node[right] {$s$};
				\draw[->] (0,-4) -- (0,4) node[above] {$t$};
				
				\node[below right] at (0,0) {$(0,0)$};
		\draw[thick,domain=-3:3,smooth,samples=100]
				plot ({\x*\x/2-3}, \x)
				node[right] {$\partial \Lambda(\Omega_a)$};
                       \draw[dotted] (-1.5,3) circle[radius=1.1180];
				\fill[black] (-3,0) circle(2pt);
				\fill[black] (-1,2) circle(2pt);
                                           \fill[black] (-1.5,3) circle(2pt);
                                           \node at (-1, 2) [below right]{$\la_0^2$};
				\node at (-3, 0) [below right] {$\left(-a^2, 0 \right)$};	
                     \draw[thick,domain=2:3,smooth,samples=100]
				plot (-\x/2, \x) node[right]{$z_0$};
                     \draw[thick,domain=-3:1,smooth,samples=100]
                               plot(\x, \x/2+2.5);	
\node at (-1, 1) [below]{$\Lambda(\Omega_a)$};
                                         \fill[pattern=dots, pattern color=blue]
    plot[domain=-3:3] (\x*\x/2-3, \x) -- 
    plot[domain=-3:3] (1.5, \x);
	\end{tikzpicture}

\begin{tikzpicture} \label{Diagram 3}
				\draw[->] (-3,0) -- (6,0) node[right] {$s$};
				\draw[->] (0,-6) -- (0,2) node[above] {$t$};
				
				\node[above right] at (0,0) {$(0,0)$};
		\draw[thick,domain=-6:0,smooth,samples=100]
				plot (\x*\x/2+3*\x+3, \x);
                     \draw[thick,domain=-1:0,smooth,samples=100]
				plot (-\x/2, \x);	

 \fill[pattern=dots, pattern color=blue]
    plot[domain=-6:0] (\x*\x/2+3*\x+3, \x) -- 
    plot[domain=-6:0] (3, \x);
\fill[black] (0.5,-1) circle(2pt);
\node at (0.3, -0.5) [right]{$\beta$};
\node at (1, -3)[right]{$\Gamma$};
\draw[thick,domain=-2:3,smooth,samples=100] plot(\x, \x/2-5/4);

	\end{tikzpicture}
\end{center}
    \caption{Point spectrum of $\D$ on $X_a$ and its translation by $z_0$.}
\end{figure}
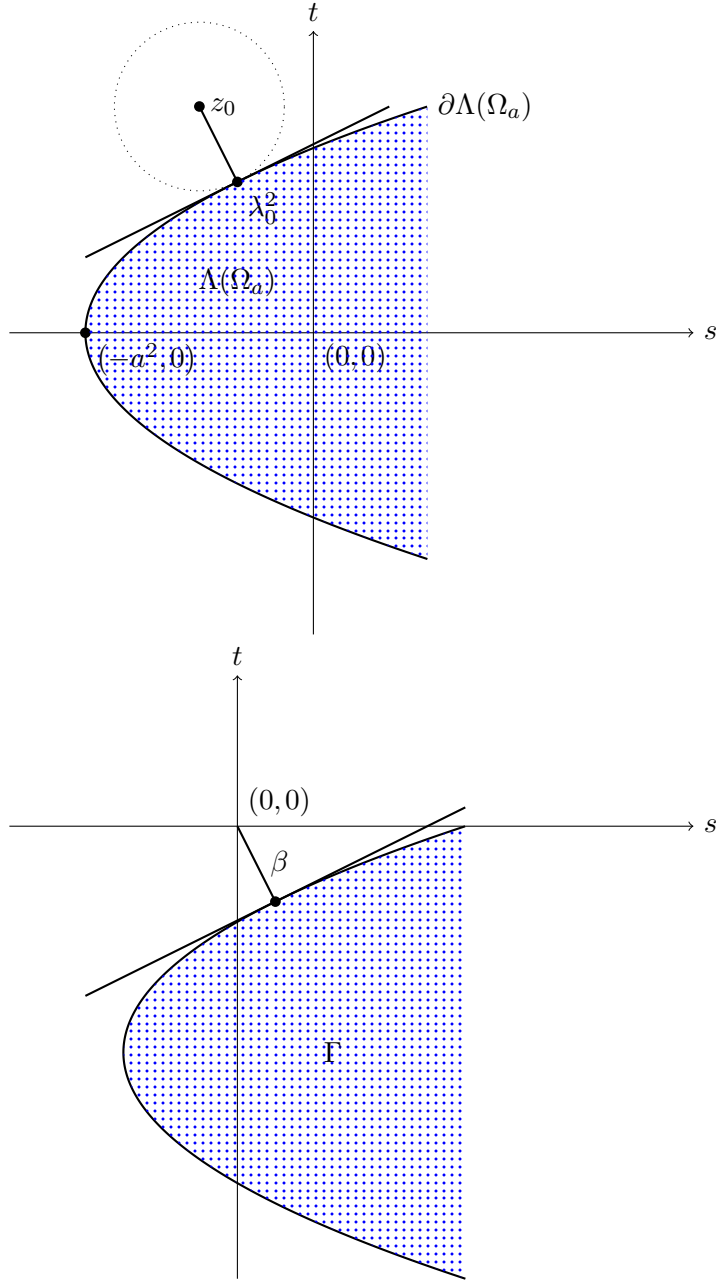

		Step 2: We shall now prove $N =0$ so that $(\la^2-\la_0^2)S_0 =0$ or equivalently $(\la^2-z_0)S_0 = (\la_0^2 - z_0)S_0$, where $S_0 = \Hc{T_0}$. It follows from \eqref{eq03} that 
		\begin{equation*}
			Span\{S_0, S_1, \ldots\}= Span\{S_0,(\la^2-z_0)S_0,\ldots, (\la^2-z_0)^N S_0\}= Span\{S_0,S_1, \ldots,S_N\}
		\end{equation*}
		Suppose that $(\la^2-\la_0^2)S_0 \neq 0$. Let $k_0$ be the largest positive integer such that $(\la^2-\la_0^2)^{k_0}S_0 \neq 0$. Then $k_0 \leq N$. Let $S=(\la^2 -\la_0^2)^{k_0-1}S_0 \in Span \{S_0,S_1, \ldots,S_N\}$. Therefore we assume $S = \sum_{j=1}^N a_jS_j$. Then 
		\begin{equation}\label{eqn8}
			(\la^2-\la_0^2)^2 S=0\quad \mbox{and} \quad (\la^2-\la_0^2)S \neq 0.
		\end{equation}
		
		Now, using \eqref{eqn8} and binomial expansion, we obtain \begin{equation*}
			(\la^2 -z_0)^k S= ((\la^2-\la_0^2)+(\la_0^2-z_0))^kS
			= (\la_0^2-z_0)^k S+ k (\la_0^2-z_0)^{k-1}(\la^2-\la_0^2)S.
		\end{equation*} 
		
		Hence for any $\phi \in H_e(\Omega_a)$,
		\begin{equation}\label{eqn9}
			|\langle (\la^2-\la_0^2)S, \phi  \rangle| \leq \frac{1}{k(\la_0^2-z_0)^{k-1}} |\langle (\la^2 -z_0)^k S, \phi \rangle | + \frac{1}{k} (\la_0^2-z_0)|\langle S, \phi \rangle|.
		\end{equation}
		Consider 
		\begin{align*}
			|\langle (\la^2-z_0)^k S, \phi \rangle| & =\left|\left\langle (\la^2-z_0)^k \sum_{j=0}^N a_jS_j, \phi \right\rangle \right| =\left |\left \langle \sum_{j=0}^N a_jA^kS_{j+k}, \phi \right\rangle \right|\\
			& = |A|^k\left|\left \langle \sum_{j=0}^N a_j S_{j+k}, \phi \right \rangle \right | \leq |\la_0^2-z_0|^k \sum_{j=0}^N \left |a_j\right| \left|\left \langle S_{j+k}, \phi \right\rangle \right| \\
			& \leq M |\la_0^2-z_0|^k \mu(\phi) \sum_{j=0}^N |a_j|.
		\end{align*}
		
		Thus \eqref{eqn9} and above inequality implies that,
		\begin{equation*}
			|\langle (\la^2-\la_0^2)S, \phi  \rangle| \leq \frac{M}{k} (\la_0^2-z_0) \mu(\phi) \sum_{j=0}^N |a_j| + \frac{1}{k} (\la_0^2-z_0)|\langle S, \phi \rangle|,
		\end{equation*}as the right hand side goes to 0 as $ k \rightarrow \infty$, we get $(\la^2-\la_0^2)^{k_0}S = 0$, which contradicts the assumption on $k_0$ from  \eqref{eqn8}. This proves  $N = 0$ and hence we get $ (\la^2-\la_0^2)S_0 =0 $. Finally by spherical Fourier inversion we obtain $\D T_0 = \la_0^2 T_0$.
		
		Next we shall prove the theorem when $T_k$'s are non-radial. Initially we shall see that given any sequence of non-radial exponential type tempered distributions $T_k$'s satisfying the conditions as mentioned in the theorem, then for any $y \in \R^d$, the corresponding sequence ${R\ell_yT_k}$ of radial distributions also satisfies the hypothesis. Above observation is an immediate consequence of the fact that $\D$ commutes with radialization and translations and thus we have $\D R(\ell_yT_k) = A R(\ell_yT_{k+1})$. It remains to show that for the seminorm $\gamma$ of $S^a(\R^d)$ in the hypothesis of the theorem and $\psi_1 \in \sar(\R^d)$, $$ | \langle R(\ell_yT_k), \psi_1 \rangle | \leq C_y M \gm(\psi_1).$$ For that consider any $\psi \in S^a(\R^d)$, 
		\begin{align*}
			\gm(\ell_y \psi) = \sup_{x \in \R^d} (1+ |x|)^m e^{a|x|} D^\ap \psi(x-y) & \leq \sup_{z \in \R^d} (1+ |y|+ |z|)^m e^{a(|y|+|z|)} D^\ap \psi(z) \\
			& \leq (1+|y|)^m e^{a|y|} \sup_{z \in \R^d} (1+|z|)^m e^{a|z|} D^\ap \psi(z) \\ & \leq C_y \gm(\psi)
		\end{align*}
		Since $|\langle T_k, \psi \rangle| \leq M \gm(\psi)$ for any $\psi \in S^a(\R^d),$ it follows that for any $\psi_1 \in \sar(\R^d)$\\
		$ |\langle R(\ell_y T_k), \psi_1 \rangle| = | \langle \ell_y T_k, \psi_1 \rangle | = | \langle T_k, \ell_{-y}\psi_1 \rangle| \leq M\gm (\ell_{-y}\psi_1) \leq C_{y^{-1}} M \gm (\psi_1) $.
		From the result proved for radial distributions, we conclude that $\D R(\ell_y T_0) = \la_0^2 R(\ell_y T_0)$, for every $y \in \R^d$.
		Again using the fact that $\D$ commutes with radializations and translations, we have $R\ell_y(\D T_0) = R\ell_y(\la_0^2 T_0)$, for every $y \in \R^d$.\\
	It is enough to prove that for any $T\in S^a(\R^d)'$, if $R(\ell_xT)=0$ for all $x \in \R^d$, then $T$ equals zero as a distribution. Indeed, if $R(\ell_xT)=0$ for all $x \in \R^d$, then $ \langle \ell_x T, h_t \rangle = 0$ for all $t > 0$ where $ h_t$ denotes the heat kernel, which is a radial function given by $ h_t(x) = \frac{1}{(4\pi t)^{d/2}} e^{\frac{-|x|^2}{4t}}$, for $t>0$. That is, $T * h_t \equiv0. $ But $T * h_t \to T$ as $t \to 0$ in the sense of distributions. Therefore $T=0$ and hence $\D T_0= \la_0^2 T_0$.

		Proof of part (b). Let $|A| < |\la_0^2-z_0|$.
		Already we have $|\la^2-z_0| \geq |\la_0^2-z_0|$, for every $\la \in \Omega_a$.
		Also we have for $\phi \in H_e(\Omega_a)$\\
		\begin{equation}
			|\langle \Hc{T_0}, \phi \rangle| \leq M \mu \bigg[\left(\frac{A}{\la^2-z_0}\right)^k \phi \bigg]
		\end{equation}
		But for  $\la \in \Omega_a$, \\
		\begin{align*}
			\bigg|\frac{A}{\la^2-z_0} \bigg| = \bigg| \frac{A}{\la_0^2-z_0} \bigg| \bigg| \frac{\la_0^2-z_0}{\la^2-z_0} \bigg| <1
		\end{align*}
		Therefore, $\mu \bigg[\left(\frac{A}{\la^2-z_0}\right)^k \phi \bigg] \to 0,$ as $k \to \infty$, which implies that $ | \langle \Hc T_0, \phi \rangle| =0,$ for every $\phi \in H_e(\Omega_a)$. Hence $T_0=0.$

		Proof of part(c). Consider $|A|> |\la_0^2 - z_0|$. It can be seen geometrically that, there exists distinct $\la', \la'' \in \Lambda(\Omega_{a})$ with $|(\la')^2 - z_0| = |A| = |(\la'')^2 - z_0|.$
		Now $(\la')^2 - z_0 = A e^{i\theta_1}$ and $(\la'')^2 - z_0 = A e^{i\theta_2}$ for $\theta_1, \theta_2 \in \R$, where $\theta_1 \neq \theta_2$. Define $$T_k = e^{ik\theta_1} \phi_{\la'} + e^{ik\theta_2} \phi_{\la''} $$ for $k \in \Z^{+}.$ Routine calculations shows that $T_k \in S^a(\R^d)'$ and $T_k$ satisfies hypothesis of main theorem \ref{Theorem 2.1}, but $T_0$ is not an eigendistribution of $\D$. 
	\end{proof}

\subsection{Proof of Theorem \ref{Theorem1.5}}

Having characterized the eigendistributions of the Laplacian, our task of proving Theorem \ref{Theorem1.5} can be accomplished once we establish the fact that functions $f$ on $\R^d$ satisfying the property that $ |f(x)| \leq M e^{a |x|} $ for a constant $M>0$ and for every $x \in \R^d$ are themselves exponential tempered distributions of type $a$, which can be seen as follows.

	For $g \in S^a(\R^d)$, 
\begin{align*}
	|\langle f,g \rangle| &=\left |\int_{\R^d} f(x)g(x)dx \right|\\
	& \leq M\left (\int_{\R^d} \frac{(1+|x|)^m e^{a|x|} |g(x)|}{(1+|x|)^m} dx \right )  \\ & \leq M' \gm_{m, 0}(g)   \Big(\int_{\R^d}\frac{1}{(1+|x|)^{m}} dx \Big).
\end{align*} The integral in the last line of the above is finite for sufficiently large $m \geq \frac{d}{2}$. Hence $f \in S^a(\R^d)'$ and thus Theorem \ref{Theorem1.5} is a consequence of Theorem \ref{Theorem 2.1}.

\section{Roe's characterization of exponential function} \label{section3}
In this section, we shall prove a characterization \ref{Theorem 1.2} for the eigenfunctions of the ordinary derivative $\frac{d}{dx}$ having exponential growth under the same framework as we have defined in section \ref{section2}.\\ 


Now in order to prove Theorem \ref{Theorem 1.2}, we shall use the same strategy used in proving Theorem \ref{Theorem1.5}. In view of the geometry of point spectrum of $\frac{d}{dx}$ on $X_a$ as depicted in Figure 3, we shall initially formulate a generalized version of Theorem \ref{Theorem 1.2} for the exponential tempered distributions of type $a$ on $\R$ as given below.
\begin{figure}[h!]
\begin{center}
	\begin{tikzpicture} \label{Diagram 4}
				\draw[->] (-3,0) -- (3,0) node[right] {$s$};
				\draw[->] (0,-3) -- (0,3) node[above] {$t$};
				
				\node[below right] at (0,0) {$(0,0)$};
\draw[thick, domain=-2:2,smooth,samples=100] plot(\x,1.5);
\draw[thick, domain=-2:2,smooth,samples=100] plot(\x,-1.5);
\fill[pattern=dots, pattern color=blue] 
plot[domain=-2:2] (\x,1.5) 
-- plot[domain=-2:2] (\x, -1.5)
-- plot[domain=-1.5:1.5] (2,\x)
-- plot[domain=-1.5:1.5] (-2,\x);
\fill[black] (0, 1.5) circle(2pt);
\fill[black] (0, -1.5) circle(2pt);
\node at (0,1.5)[above right]{$ia$};
\node at (0, -1.5)[below right]{$-ia$};
\fill[black] (1,1.5) circle(2pt);
\node at (1,1.5)[above]{$\lambda_0$};
		
	\end{tikzpicture} \begin{tikzpicture} \label{Diagram 5}
\draw[->] (-3,0) -- (3,0) node[right] {$s$};
				\draw[->] (0,-3) -- (0,3) node[above] {$t$};
				
				\node[below right] at (0,0) {$(0,0)$};
\draw[thick, domain=-2:2,smooth,samples=100] plot(1.5,\x);
\draw[thick, domain=-2:2,smooth,samples=100] plot(-1.5,\x);
\fill[pattern=dots, pattern color=blue] 
plot[domain=-2:2] (1.5,\x) 
-- plot[domain=-2:2] (-1.5,\x)
-- plot[domain=-1.5:1.5] (\x,2)
-- plot[domain=-1.5:1.5] (\x, -2);
\fill[black] (1.5, 0) circle(2pt);
\fill[black] (-1.5,0) circle(2pt);
\node at (1.5,0)[above right]{$a$};
\node at (-1.5, 0)[below right]{$-a$};
\fill[black] (-1.5,1) circle(2pt);
\node at (-1.5,1)[right]{$i\lambda_0$};
\draw[thick,domain=-2.5:-1.5,smooth,samples=100]plot(\x,1);
\node at (-2.5,1)[left]{$i\lambda_0-\ap$};
\fill[black] (-2.5,1) circle(2pt);

	\end{tikzpicture}
\end{center}

\caption{The strip $\Omega_a$ and the point spectrum of $\frac{d}{dx}$ on $X_a$.}
\end{figure}
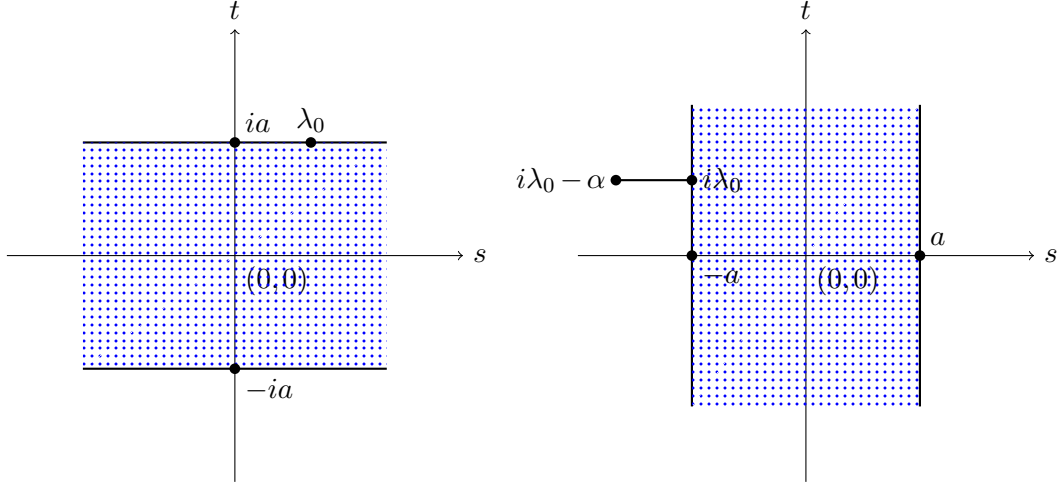

\begin{theorem}\label{Theorem 4.1}
	\par 	For any $\la_0 \in \C$ with $a=|\im(\la_0)| \neq 0$, let $\{T_k\}_{k \in \mathbb{Z}}$ be a doubly infinite sequence of exponential tempered distributions of type $a$ on $\R$ satisfying 
\begin{enumerate}
	\item For some non-zero $A \in \C $ and $\alpha>0$, we have $\begin{cases}(\frac{d}{dx} - i\la_0 + \ap)T_k = AT_{k+1},&\mbox{ if}~ \im(\la_0)>0,\\  (\frac{d}{dx} - i\la_0 - \ap)T_k = AT_{k+1},&\mbox{ if}~ \im(\la_0)<0,\end{cases}$
	\item  There exist a seminorm $\gm$ on    $S^a(\R)$ and a constant $M>0$ such that $|\langle T_k, f \rangle | \leq M \gm(f)$ for all $f \in S^a(\R)$ and for all $k \in \Z $. 
\end{enumerate} Then the following assertions hold.
\begin{enumerate}[(a)]
	\item If $|A| = \ap$, then $\frac{d}{dx} T_0 = i{\la_0} T_0$.
	\item If $|A| < \ap$, then $T_k = 0,$ for all $k \in \Z$. 
	\item There are solutions satisfying conditions (1) and (2) which are not eigendistributions of $\frac{d}{dx}$ when $|A| >  \ap$.
\end{enumerate}
\end{theorem}
 For the sake of brevity, we will be providing only the first few steps of the proof of above theorem, as it follows the same line of argument as in the proof of Theorem \ref{Theorem 2.1}. 
\begin{proof} Part(a). First let us assume the distributions $\{T_k\}$'s be radial. Let $|A| = |\ap|$ and $a=\im(\la_0) > 0$. In order to establish $\frac{d}{dx} T_0 = i{\la_0} T_0$, it is enough for us to show that $(i\la-i\la_0)\Hc{T_0}=0$.
So initially we must show that \begin{equation}
	(i\la-i\la_0)^{N+1}\Hc{T_0}=0.
\end{equation} It is easy to verify that, for any $\phi \in H_e(\Omega_{a})$, we obtain
$$ \langle (i\la-i\la_0)^{N+1}\Hc{T_0}, \phi \rangle \leq M \mu \bigg[\left( \frac{A}{\la- \la_0-i\ap}\right)^k (\la-\la_0)^{N+1} \phi \bigg] $$ where $\mu$ is a seminorm given by $\mu (\phi) = \sup_{\la \in \Omega_{a}}|\frac{d^\tau}{d\la^\tau} P(\la) \phi(\la)|$ for some polynomial $P(\la)$. Now for a fixed $N=6\tau +1$, let us denote $$ F^k(\la)= \left| \frac{d^\tau}{d\la^\tau} P(\la) \left(\frac{A}{\la- \la_0-i\ap}\right)^k (\la-\la_0)^{N+1} \phi \right|. $$ Our aim is to show that $\sup_{\la \in \Omega_{a}} F^k(\la) \to 0$ as $k \to \infty$. Note that

 \begin{multline} \frac{d^\tau}{d\la^\tau} \left(P(\la) \left(\frac{A}{\la- \la_0-i\ap}\right)^k (\la-\la_0)^{N+1} \phi \right) = \sum_{l+m+n= \tau} C_{lmn} \frac{d^l}{d\la^l} \left(\frac{A}{\la- \la_0-i\ap}\right)^k \\ \times \frac{d^m}{d\la^m} (\la-\la_0)^{N+1}\frac{d^n}{d\la^n} (P(\la)\phi)
\end{multline}

Using the fact that $\phi \in H_e(\Omega_{a}) $ and $\left| \frac{A}{\la- \la_0-i\ap }\right| = \left|\frac{\ap}{\la- \la_0-i\ap} \right| \leq 1$ for $ \la \in \Omega_{a}$, we obtain the following two inequalities where $C_1$ and $C_2$ are some arbitrary constants. 
\begin{equation}
	F^k(\la) \leq C_1 k^{\tau} \left|\frac{\ap}{\la- \la_0-i\ap} \right|^k |\la-\la_0|^{N+1-\tau}
\end{equation}

\begin{equation}\label{eq15}
	F^k(\la) \leq C_2 k^{\tau} \left|\frac{\ap}{\la- \la_0-i\ap} \right|^k
\end{equation} 
Now consider a compact connected neighborhood $\mathcal{U}$ of $\la_0$ in $\Omega_{a}$ such that if $\la \notin \mathcal{U}$, then $\left|\frac{\ap}{\la- \la_0-i\ap} \right| < \frac{1}{2}$. It follows from equation \ref{eq15}, that $F^k(\la) \to 0$ uniformly in $\Omega_{a} \setminus \mathcal{U}$. Next we shall show that $ \sup_{\la \in \mathcal{U}} k^{\tau}  \left|\frac{\ap}{\la- \la_0-i\ap} \right|^k |\la-\la_0|^{5\tau+2} \to 0$, as $k \to \infty$.
Since  $$\sup_{\la \in \mathcal{U}} k^{\tau}  \left|\frac{\ap}{\la- \la_0-i\ap} \right|^k |\la-\la_0|^{5\tau+2} =  \sup_{\la \in \mathcal{U}} k^{\tau}  \left|\frac{\la_0 -(\la_0+i\ap)}{\la- (\la_0+i\ap)} \right|^k \left|[\la-(\la_0+i\ap)] - [\la_0-(\la_0+i\ap)]\right|^{5\tau+2}, $$ this is equivalent to prove that $$\sup_{z \in \Omega} k^{\tau} \left| \frac{\bt}{z} \right|^k |\bt-z|^{5\tau+2} \to 0, $$ where $\Omega$ is a rectangular compact region containing $\bt$.\\   Rest of the proof (including Part (b) and Part (c)) follows simliarly as in the proof of Theorem\ref{Theorem 2.1}.
\end{proof}

\subsection{Proof of Theorem \ref{Theorem 1.2}}

Since we already established that any function on $\R$ satisfying the condition: $|f(x)| \leq M e^{a|x|} $ for all $x\in \R$, are exponential tempered distributions of type $a$ on $\R$, it is straight forward to see that Theorem \ref{Theorem 1.2} is an immediate consequence of Theorem \ref{Theorem 4.1}.




\section*{Acknowledgment}
The first author would like to thank UGC-CSIR, India for providing the fellowship for the completion of this work.


\begin{thebibliography}{}
	
%
\bibitem{natan} {Y. Ben Natan, Y. Weit}, \textit{Integrable harmonic functions on $\R^n$}, {J.  Funct. Anal.}, {150}, {1997}, {471-477}.	


	\bibitem{Howard} {Howard, Ralph and Reese, Margaret}, 	\textit{Characterization of eigenfunctions by boundedness conditions}, {Canadian Mathematical Bulletin}, {35}, {1992}, {204--213}.
	
	\bibitem{howard} {Howard, Ralph}, \textit{A note on {R}oe's characterization of the sine function}, {Proceedings of the American Mathematical Society}, {105}, {1989}, {658--663}.


\bibitem{pkumar}
	{Kumar, Pratyoosh and Ray, Swagato K. and Sarkar, Rudra P.}, 	\textit{Characterization of almost {$L^p$}-eigenfunctions of the {L}aplace-{B}eltrami operator}, {Transactions of the American Mathematical Society}, {366}, {2014}, {3191--3225}.
	
	\bibitem{muna} Naik, M., Sarkar, R. P., \textit{Characterization of eigenfunctions of the Laplace-Beltrami operator using Fourier multipliers}. J. Funct. Anal. 279,  2020.
	
	\bibitem{niki}{A. V. Nikiforov and V.B. Uvarov}, \textit{Special functions of mathematical physics: A unified introduction with applications}, {Birkh\"auser}, {1988}.
	
\bibitem{Roe} 
	{Roe, J.},
	\textit{A characterization of the sine function},
	{Mathematical Proceedings of the Cambridge Philosophical
		Society}, {87}, {1980}, {69--73}.
	


\bibitem{Str}
	{Strichartz, Robert S.},
	\textit{Characterization of eigenfunctions of the {L}aplacian by
		boundedness conditions},
	{Transactions of the American Mathematical Society},
	{338}, {1993}, {971--979}.	
	
	
	
\end{thebibliography}
\end{document}